\documentclass[final]{siamltex}

\usepackage{graphicx}
\usepackage{amsmath,amsfonts,mathrsfs,amssymb,cite}
\usepackage[usenames]{color}

\newtheorem{prop}{Proposition}[section]

\numberwithin{equation}{section}

\newcommand{\abs}[1]{\left\vert#1\right\vert}

\title{\bf On Quasi-static Cloaking Due to Anomalous Localized Resonance in $\mathbb{R}^3$}

\author{
Hongjie Li\thanks{School of Mathematics and Statistics, Beijing Institute of Technology, Beijing, 100081, P.~R.~China.}
\and
Jingzhi Li\thanks{Department of Mathematics, South University of Science and
Technology of China, Shenzhen 518055, P.~R.~China. Email: {\tt li.jz@sustc.edu.cn}}
\and Hongyu Liu\thanks{Department of Mathematics, Hong Kong Baptist University, Kowloon Tong, Hong Kong SAR, and HKBU Institute of Research and Continuing Education, Virtual University Park, Shenzhen, P. R. China.   Email:  {\tt hongyu.liuip@gmail.com}} }

\begin{document}

\maketitle

\begin{abstract}
This work concerns the cloaking due to anomalous localized resonance (CALR) in the quasi-static regime. We extend the related two-dimensional studies in \cite{Ack13,Klsap} to the three-dimensional setting. CALR is shown not to take place for the plasmonic configuration considered in \cite{Ack13,Klsap} in the three-dimensional case. We give two different constructions which ensure the occurrence of CALR. There may be no core or an arbitrary shape core for the cloaking device. If there is a core, then the dielectric distribution inside it could be arbitrary.
\end{abstract}



\pagestyle{myheadings}
\thispagestyle{plain}
\markboth{H. LI, J. LI AND H. LIU}{Cloaking Due to Anomalous Localized Resonance}

\section{Introduction}

This work concerns the invisibility cloaking due to anomalous localized resonance (CALR) in the quasi-static regime, which has gained growing interest in the literature; see \cite{Acm13,Ack13,Ack14,AK,Bos10,Brl07,CKKL,KLO,Klsap,Min06,Nmm94} and the references therein. Let $\Sigma$ and $\Omega$ be bounded domains in $\mathbb{R}^d$, $d=2, 3$, such that $\Sigma\Subset \Omega$. $\Sigma$, $\Omega\backslash\overline{\Sigma}$ and $\mathbb{R}^d\backslash\overline{\Omega}$ signify, respectively, the core, shell and matrix of a cloaking device, which hosts a dielectric object as follows
\begin{equation}\label{eq:struc}
   \epsilon(x) = \begin{cases}
   \epsilon_c(x),\quad & x\in\Sigma,\\
   \epsilon_s(x),\quad & x\in \Omega\backslash\overline{\Sigma},\\
   \epsilon_m(x),\quad & x\in\mathbb{R}^d\backslash\overline{\Omega}\,.
   \end{cases}
\end{equation}
In most of the existing studies, one takes $\epsilon_m\equiv 1$ and $\epsilon_c\equiv 1$. $\epsilon_s$ is negatively valued, which denotes the plasmonic material parameter. Let $\eta\in\mathbb{R}_+$ denote the loss parameter and consider a material distribution given as
\begin{equation}\label{eq:d1}
\epsilon_\eta(x)=\epsilon(x)+i\eta \chi_D(x),\quad x\in\mathbb{R}^d,
\end{equation}
where $\epsilon$ is given in \eqref{eq:struc}, and $\chi$ denotes the characteristic function of the domain $D$, with $D=\mathbb{R}^d$ or $D=\Omega\backslash\overline{\Sigma}$. For time-harmonic wave propagation in the quasi-static regime, the wave pressure $u(x)=u_\eta(x)\in\mathbb{C}$ satisfies the following equation
\begin{equation}\label{eq:m1}
\begin{cases}
\nabla_x\cdot\left(\epsilon_\eta(x)(\nabla_x u_\eta(x))\right)=f(x),\qquad & x\in\mathbb{R}^d,\\
u_\eta(x)=  o(1),\quad d=2;\ \ \mathcal{O}(|x|^{-1}),\quad d=3,\qquad & |x|\rightarrow +\infty,
\end{cases}
\end{equation}
where $f(x)$ denotes a source term that is compactly supported in $\mathbb{R}^d\backslash\overline{\Omega}$ and satisfies
\begin{equation}\label{eq:m2}
\int_{\mathbb{R}^d} f(x)\ dx=0.
\end{equation}
Define
\begin{equation}\label{eq:m3}
E_\eta=E_\eta(u_\eta, \epsilon_\eta, f):=\int_{D}\frac{\eta}{2}|\nabla_x u_\eta(x)|^2\ dx,
\end{equation}
where $u_\eta$ is the solution to \eqref{eq:m1}. $E_\eta$ denotes the rate at which the energy of the wave field is dissipated into heat. Then anomalous localized resonance (ALR) is said to occur if there holds
\begin{equation}\label{eq:cond1}
E_\eta\rightarrow +\infty\quad\mbox{as}\ \ \eta\rightarrow +0.
\end{equation}
In what follows, we sometimes simply refer to ALR as resonance. If in addition to \eqref{eq:cond1}, one further has that
\begin{equation}\label{eq:cond2}
|u_\eta(x)/\sqrt{E_\eta}|\rightarrow 0\quad\mbox{as}\ \ \eta\rightarrow +0\quad \mbox{when}\ \ |x|>a,
\end{equation}
where $a\in\mathbb{R}+$ is such that the central ball $B_a$  contains $\Omega$, then it is said that CALR occurs. Here and also in what follows, $B_\tau$ with $\tau\in\mathbb{R}_+$ denotes a central ball of radius $\tau$ in $\mathbb{R}^d$, $d\geq 2$. By \eqref{eq:cond2}, it is readily seen that if CALR occurs, then both the source $f$ and the cloaking device $(\Omega; \epsilon_\eta)$ are invisible to the wave observation made from the outside of $B_a$. If \eqref{eq:cond1} is replaced by
\begin{equation}\label{eq:cond3}
\limsup_{\eta\rightarrow +0} E_\eta=+\infty,
\end{equation}
then it is said that weak CALR occurs. We refer to \cite{Ack13}, \cite{AGJKLSW} and \cite{Klsap} for more discussions on the anomalous localized resonance and its connection to invisibility cloaking.

The anomalous localized resonance phenomenon was first observed in \cite{Nmm94} and connected to invisibility cloaking in \cite{Min06}. Recently, a mathematical theory was developed in \cite{Ack13} by Ammari et al to rigorously explain the CALR observed in \cite{Nmm94} and \cite{Min06}. In their study, $d=2$ and $(\Omega; \epsilon_\eta)$ is given by \eqref{eq:d1} with
\begin{equation}\label{eq:ammari}
\epsilon_m\equiv 1, \quad \epsilon_c\equiv 1,\quad \epsilon_s\equiv -1\quad\mbox{and}\quad D=\Omega\backslash\overline{\Sigma}.
\end{equation}
Moreover, they let $\Omega=B_{r_e}$ and $\Sigma=B_{r_i}$, with $0<r_i<r_e$.  For the above plasmonic configuration, the solution $u_\eta$ to \eqref{eq:m1} in \cite{Ack13} was shown to have a spectral representation associated with a Neumann-Poincar\'e-type operator. Using the spectral representation, it is shown that there exists a critical radius
\begin{equation}\label{eq:critical}
r^*=\sqrt{r_e^3/r_i},
\end{equation}
such that when a generic source term $f$ lies within $B_{r^*}\backslash B_{r_e}$, then CALR occurs; and when $f$ lies outside $B_{r^*}$, then CALR does not occur. Here and also in what follows, when we say that CALR does not occur, it actually means that weak CALR does not occur. Later on, the CALR was considered from a variational perspective in \cite{Klsap} by taking $d=2$ and
$(\Omega; \epsilon_\eta)$ in \eqref{eq:d1} with
\begin{equation}\label{eq:kohn}
\epsilon_m\equiv 1, \quad \epsilon_c\equiv 1,\quad \epsilon_s\equiv -1\quad\mbox{and}\quad D=\mathbb{R}^2,\quad \Omega=B_{r_e}.
\end{equation}
By using the primal and dual variational principles, it is shown in \cite{Klsap} that for a large class of sources, if $\Sigma=\emptyset$, then resonance always occurs; whereas if $\Sigma\neq \emptyset$ and $\Sigma\subset B_{r_i}$, then there exists the critical radius $r^*$ in \eqref{eq:critical} for the occurrence and nonoccurrence of resonance.

The aim of this work is to extend the related results in \cite{Ack13,Klsap} to the three dimensional setting. Indeed, the three-dimensional CALR was considered in the literature and the situation becomes much more complicated. In \cite{Ack14}, it is shown that if one takes a similar structure as that in \eqref{eq:ammari} but with $d=3$, then CALR does not occur. The same conclusion was draw in \cite{KLO} without the quasi-static approximation. In \cite{Acm13}, the anomalous localized resonance is shown to take place by using a folded geometry where the plasmonic material $\epsilon_s$ is spatially variable. In this paper, we show that CALR does not occur for the configuration \eqref{eq:kohn} in $\mathbb{R}^3$. Then, we show that by properly choosing the plasmonic parameters, CALR can still happen, at least approximately. We follow both the spectral and variational arguments developed, respectively, in \cite{Ack13} and \cite{Klsap}.

The rest of the paper is organized as follows. In Sections 2, by using the variational argument, we show the nonoccurrence and occurrence of ALR by taking the loss parameter to be given over the whole space $\mathbb{R}^3$. In Section 3, by following the relevant study in \cite{Ack13}, we consider the occurrence and nonoccurrence of CALR by taking the loss parameter to be given only in the plasmonic layer.

\section{Variational perspective on ALR in three dimensions}\label{sect:2}

Henceforth, we let $\epsilon_m\equiv 1$ and $\eta\in\mathbb{R}_+$ be a constant. Throughout the present section, we let
the source term $f$ be a real-valued distributional functional such that it is supported at a distance $q$ from the origin, and has a zero mean :
\begin{equation}\label{eq:source}
f=F\, \mathcal{H}^2 \lfloor \partial B_q, \quad F: \partial B_q \rightarrow \mathbb{R},\ \ F \in L^2(\partial B_q)\ \ \mbox{and}\ \ \int_{\partial B_q} F d \mathcal{H}^2=0,
\end{equation}
where $\mathcal{H}^2\lfloor\partial B_q$ denotes the two-dimensional Hausdorff measure restricted to the set $\partial B_q$. Moreover, in this section, we let
\begin{equation}\label{eq:d2}
\mbox{$(\Omega; \epsilon_\eta)$ be given by \eqref{eq:d1} with $\Omega=B_{r_e}$ and $D=\mathbb{R}^3$, }
\end{equation}
and without loss of generality, it is assumed that $r_e>1$ and $r_i=1$. Indeed, the subsequent results derived in this section can be easily extended to the general case $0<r_i<r_e<+\infty$ by a direct scaling argument. For the solution $u_\eta\in\mathbb{C}$ to \eqref{eq:m1} with $\epsilon_\eta$ given in \eqref{eq:d2}, we set
\begin{equation}\label{eq:decom}
  u_{\eta} = v_{\eta} +i \frac{1}{\eta}w_{\eta} \quad \mathrm{with} \quad  v_{\eta},w_{\eta}: \mathbb{R}^3 \rightarrow \mathbb{R}.
\end{equation}
It is straightforward to verify that
\begin{align}
&\nabla \cdot ( \epsilon \nabla v_{\eta}) -  \triangle w_{\eta}  = f\qquad\ \ \mbox{in\ \ } \mathbb{R}^3,\label{eq:change1}\\
&\nabla \cdot ( \epsilon \nabla w_{\eta}) + \eta^2 \triangle v_{\eta}  = 0\qquad\mbox{in\ \ } \mathbb{R}^3.\label{eq:change2}
\end{align}
Accordingly, the energy $E_{\eta}(u_{\eta})$ can be represented with $v_{\eta}$ and $w_{\eta}$ as
\begin{equation}
 E_{\eta}(u_{\eta}) = \frac{\eta}{2} \int_{\mathbb{R}^3} |\nabla u_{\eta}|^2 =
 \frac{\eta}{2} \int_{\mathbb{R}^3} |\nabla v_{\eta}|^2 +  \frac{1}{2\eta} \int_{\mathbb{R}^3} |\nabla w_{\eta}|^2\,.
\end{equation}

The following variational principles were proved in \cite{Klsap} when $d=2$, and can be extended to the three-dimensional case for the present study by straightforward modifications. Define
\begin{equation}\label{eq:v1}
I_\eta(v, w):=\frac{\eta}{2}\int_{\mathbb{R}^3} |\nabla v|^2+\frac{1}{2\eta}\int_{\mathbb{R}^3}|\nabla w|^2,\quad v, w\in H_{loc}^1(\mathbb{R}^3).
\end{equation}
Consider the optimization problem
\begin{equation}\label{eq:pv}
\begin{split}
& \min_{(v, w)\in H_{loc}^1(\mathbb{R}^3)\times H_{loc}^1(\mathbb{R}^3)} I_\eta (v, w)\\
&\qquad \mbox{subject to}\ \nabla\cdot (\epsilon\nabla v)-\Delta w=f,
\end{split}
\end{equation}
where $v$, $w$ are assumed to be real-valued. \eqref{eq:pv} is referred to as the primal variational problem, and the minimizing pair is attainable at $(v_\eta, w_\eta)\in H_{loc}^1(\mathbb{R}^3)\times H_{loc}^1(\mathbb{R}^3)$ such that $u_\eta=v_\eta+i\eta^{-1}w_\eta$ is a solution to \eqref{eq:m1}. Similarly, we define
\begin{equation}\label{eq:dvf}
J_\eta(v, \psi):=\int_{\mathbb{R}^3} f\cdot \psi-\frac{\eta}{2}\int_{\mathbb{R}^3}|\nabla v|^2-\frac{\eta}{2}\int_{\mathbb{R}^3}|\nabla \psi|^2,\quad v, w\in H_{loc}^1(\mathbb{R}^3),
\end{equation}
and consider the following optimization problem
\begin{equation}\label{eq:dv}
\begin{split}
& \max_{(v, \psi)\in H_{loc}^1(\mathbb{R}^3)\times H_{loc}^1(\mathbb{R}^3)} J_\eta (v, \psi)\\
&\qquad \mbox{subject to}\ \nabla\cdot (\epsilon\nabla \psi)+\eta \Delta v=0,
\end{split}
\end{equation}
where $v, \psi$ are assumed to be real-valued. \eqref{eq:dv} is referred to as the dual variational problem, and the maximizing pair is attainable at $(v_\eta, \psi_\eta)\in H_{loc}^1(\mathbb{R}^3)\times H_{loc}^1(\mathbb{R}^3)$ such that $u_\eta=v_\eta+i \psi_\eta$ is a solution to \eqref{eq:m1}.

We shall make use the variational principles introduced above to prove the resonance and non-resonance results. In doing so, the spherical harmonic functions $Y_k^l(\hat x)$ for $\hat x\in\mathbb{S}^2$, $k\in\mathbb{N}\cup\{0\}$ and $l=-k, \ldots, k$ will be needed and they form an orthonormal basis to $L^2(\mathbb{S}^2)$; see \cite{CK}. In the rest of the current section, for ease of notations, we write $Y^k$ instead of $Y_k^l$ to signify the spherical harmonic functions of order $k$. Set
\[
f_k^q(x):=Y^k \mathcal{H}^2\lfloor \partial B_q,\quad x\in\partial B_q.
\]
Hence, the source $f$ in \eqref{eq:source} can be written as
\begin{equation}\label{eq:source2}
f(x)=\sum_{k=1}^\infty \alpha_k f_k^q(x),\quad \alpha_k=\int_{\mathbb{S}^2} f(q\hat x)\cdot \overline{Y^k}(\hat x)\ ds(\hat x),\quad x\in\partial B_q,
\end{equation}
where and also in what follows, for $x\in\mathbb{R}^3$ and $x\neq 0$, $x=|x|\cdot \frac{x}{|x|}:=r\cdot \hat{x}$ denotes the spherical coordinates. Moreover, in the subsequent arguments, we let $C$ and $\widetilde{C}$ denote two generic positive constants that may change from one inequality to another, but should be clear from the context.

\smallskip

The following proposition will be needed and can be proved by direct verifications.

\smallskip

\begin{prop}\label{prop:1}
Consider the PDE for $\psi$ : $\mathbb{R}^3 \rightarrow \mathbb{C}$,
\begin{equation}\label{eq:test1}
 \nabla_x \cdot (A(x) \nabla_x \psi(x)) = 0, \qquad \psi(x)=\mathcal{O}(|x|^{-1})\ \ \mbox{as}\ \ \abs{x} \rightarrow \infty.
\end{equation}
where
\begin{equation}
   A(x) = \left \{
       \begin{array}{c c}
         -1-\frac{1}{k}, &\ \ r < r_e,\\
         +1,             &\ \ r \geq r_e,
       \end{array}
       \right.
\end{equation}
with any $k \in \mathbb{N}$. Then there exists a non-trivial solution $\psi = \hat{\psi}_k\in H_{loc}^1(\mathbb{R}^3)$ which achieves its maximum value at a
point with $r = r_e $, given by
\[
     \hat{\psi}_{k}(x) = \left \{
       \begin{array}{c c}
         r^k\cdot Y^k(\hat x)              & \mbox{for}\ \ r < r_e, \\
         \\
         r^{-k-1} \cdot {r_e}^{2k+1}\cdot Y^k(\hat x)  & \mbox{for}\ \ r \geq r_e.
       \end{array}
       \right.
\]
Moreover, one has that
\begin{equation}\label{eq:int1}
   \int_{\mathbb{R}^3}{\abs{\nabla \hat{\psi}_k}^2} =- \int_{\abs{x} = r_e}{\overline{\hat{\psi}_k}\cdot \bigg[\frac{\partial \hat{\psi}_k} {\partial r}}\bigg]
   =C(k+2){r_e}^{2k+1},
\end{equation}
where $[\cdot ]$ denotes the jump of the normal flux of the function $\hat{\psi}_k$ across $|x|=r_e$.
\end{prop}

\subsection{Non-resonance result}

In this subsection, we consider the non-resonance for the standard plasmonic configuration in \eqref{eq:d2} with $\Sigma=B_1$, $\Omega=B_{r_e}$ and $\epsilon_c=1$, $\epsilon_s=-1$. This is exactly the one considered in \cite{Klsap} for the two-dimensional case. We have

\begin{theorem}\label{thm:nr1}
Let $(\Omega; \epsilon_\eta)$ be given in \eqref{eq:d2} with $\Sigma=B_1$, $\Omega=B_{r_e}$ and $\epsilon_c=1$, $\epsilon_s=-1$. Let $f$ be given in \eqref{eq:source2}. Then ALR does not occur.
\end{theorem}

\begin{proof}
We make use of the primal variational principle \eqref{eq:pv} to prove the theorem. To that end, we first construct the test functions $v_\eta$ and $w_\eta$ that satisfy the PDE constraint in \eqref{eq:pv} for $\eta\rightarrow +0$. For $k\in\mathbb{N}$, we set
\begin{equation}\label{eq:decomp1}
\hat{v}_k(x)=\begin{cases}
\ \ r^k Y^k(\hat x),\qquad & |x|\leq 1,\medskip\\
\left(\frac{1}{2k+1} r^k+\frac{2k}{2k+1}r^{-k-1} \right) Y^k(\hat x),\qquad & 1<|x|\leq r_e,\medskip\\
\left(\frac{4k+4k^2+r_e^{2k+1}}{r_e^{2k+1}(2k+1)^2}r^k+\frac{2k(-1+r_e^{2k+1})}{(2k+1)^2} r^{-k-1}\right) Y^k(\hat x),\qquad & r_e<|x|\leq q,\medskip\\
\frac{q^{2k+1}\left(1+4k(k+1)r_e^{-1-2k}\right)+2k\left(-1+r_e^{2k+1}\right)}{(2k+1)^2} r^{-k-1} Y^k(\hat x),\qquad & q<|x|.
\end{cases}
\end{equation}
It is straightforward to verify that $\hat{v}_k$ is continuous over $\mathbb{R}^3$, and $\hat{v}_k(x)$ satisfies
\[
\nabla_x\cdot(\epsilon(x)\nabla_x\hat{v}_k(x))=0\quad\mbox{for}\ \ x\in\mathbb{R}^3\backslash\partial B_q.
\]
However, across $\partial B_q$, $\hat{v}_k$ has a jump in its normal flux as follows,
\begin{equation}\label{eq:jump1}
\left[\nu\cdot\nabla\hat{v}_k\right]|_{\partial B_q}=-\frac{\left(4k(k+1)+r_e^{2k+1}\right)q^{k-1}}{r_e^{2k+1}(2k+1)} Y^k,
\end{equation}
where $\nu$ denotes the exterior unit normal vector to $\partial B_q$. Therefore, if we let
\begin{equation}\label{eq:decomp2}
\lambda_k=-\alpha_k\cdot\frac{(2k+1)\cdot q\cdot r_e}{4k(k+1)+r_e^{2k+1}}\cdot q^{-k}\cdot r_e^{2k},
\end{equation}
then one readily verifies that
\begin{equation}\label{eq:decomp3}
\nabla\cdot(\epsilon\nabla (\lambda_k\hat{v}_k))=\alpha_k\, f_k^q\qquad\mbox{in}\quad \mathbb{R}^3.
\end{equation}
Hence, by setting
\begin{equation}\label{eq:decomp4}
v_\eta=\sum_{k=1}^\infty \lambda_k\hat{v}_k,
\end{equation}
with $\lambda_k$ and $\hat{v}_k$, respectively, given in \eqref{eq:decomp2} and \eqref{eq:decomp1}, then by \eqref{eq:decomp4}, one sees that there holds
\begin{equation}\label{eq:decomp5}
\nabla\cdot(\epsilon\nabla v_\eta)=f.
\end{equation}

Finally, we let $w_\eta\equiv 0$, then by virtue of \eqref{eq:decomp5}, $(v_\eta, w_\eta)$ clearly satisfies the PDE constraint in \eqref{eq:pv}. Hence, by the primal variational principle, together with straightforward calculations (though a bit tedious), one has that
\[
\begin{split}
E_\eta(u_\eta) & \leq I_\eta(v_\eta, w_\eta)=\frac{\eta}{2}\int_{\mathbb{R}^3}|\nabla v_\eta|^2=\frac{\eta}{2}\sum_{k=1}^\infty|\lambda_k|^2\int_{\mathbb{R}^3}|\nabla\hat{v}_k|^2\\
& \leq C\eta\sum_{k=1}^\infty |\alpha_k|^2\frac{k^2}{q^{2k}}\cdot\frac{q^{2k}}{k^2}\leq C\eta \|F\|_{L^2(\partial B_q)}^2.
\end{split}
\]
That is, the resonance does not occur.

The proof is complete.
\end{proof}

\subsection{ALR with no core}

Theorem~\ref{thm:nr1} indicates that the standard plasmonic structure does not induce the ALR. In order for ALR to take place, one has to devise different plasmonic structures. Next, we first consider a construction without a core, which ensures that ALR can always occur.

\begin{theorem}\label{thm:r1}
Consider the configuration $(\Omega; \epsilon_\eta)$ described in \eqref{eq:d2} with $\Sigma=\emptyset$. Let $f$ be given in \eqref{eq:source2} and assume that $\alpha_{k_0}\neq 0$ for some $k_0\in \mathbb{N}$. Set $\epsilon_s=-1-k_0^{-1}$. Then ALR occurs.
\end{theorem}
\begin{proof}
Since $\alpha_{k_0}\neq 0$, we first assume that $\Re \alpha_{k_0}\neq 0$. Then we choose
\begin{equation}\label{eq:test11}
v_\eta=0\quad\mbox{and}\quad \psi_\eta=\lambda_\eta\Re \overline{\hat{\psi}_{k_0}}(x)
\end{equation}
where $\lambda_{\eta}$ satisfies $\Re\alpha_{k_0}\cdot\lambda_\eta>0$ and will be further chosen below. Clearly, by Proposition~\ref{prop:1}, the pair $(v_{\eta}, \psi_{\eta})$ satisfies the constraint
$\nabla \cdot (\epsilon \nabla \psi_{\eta}) + \eta \triangle v_{\eta} = 0$. By the dual variational principle, we have
\begin{equation*}\
\begin{split}
   E_{\eta}(u_{\eta}) & \geq J_{\eta}(v_{\eta}, \psi_{\eta}) = J_{\eta}( 0 , \psi_{\eta})=
  \int_{\mathbb{R}^3}{f \cdot {\psi_{\eta}}} - \frac{\eta} {2} \int_{\mathbb{R}^3}{\abs{ \nabla \psi_{\eta} }^2} \\
                      & = \Re \int_{\partial B_q} {\alpha_{k_0} Y^{k_0} \cdot \lambda_{\eta} q^{-{k_0}} {r_e}^{2{k_0}+1} \overline{Y^{k_0}} }
                        -\frac{\eta} {2} \abs{ \lambda_{\eta} }^2 \int_{\mathbb{R}^3} {\abs{ \nabla \Re \overline{\hat{\psi}_{k_0}} }^2} \\
                      & \geq \widetilde{C} \lambda_{\eta} - C(\eta \abs{\lambda_{\eta}}^2),
\end{split}
\end{equation*}
where the two positive constants $\widetilde{C}$ and $C$ depend only on $q, r_e, k_0$ and $\alpha_{k_0}$.  Choosing $\lambda_{\eta} \rightarrow +\infty$ with $\eta \abs{\lambda_{\eta}}^2 \rightarrow 0$, we obtain
$E_{\eta}(u_{\eta}) \rightarrow \infty$ for $\eta \rightarrow +0$.

Next, if $\Im \alpha_{k_0}\neq 0$, then by choosing
\[
v_\eta=0\quad\mbox{and}\quad \psi_\eta=\lambda_\eta\Im \overline{\hat{\psi}_{k_0}}(x)
\]
and using a similar argument as the previous case, one can show that resonance occurs.

The proof is complete.
\end{proof}

\subsection{Approximate ALR with an arbitrary shape core}

In this subsection, we assume that $\Sigma\subset B_1$ is a simply connected domain with a $C^2$-smooth boundary. Let us consider a configuration $(\Omega; \epsilon_\eta)$ given by \eqref{eq:d2}. We assume that $\epsilon_c(x)$, $x\in\Sigma$, is a symmetric-positive-definite matrix-valued function satisfying
\begin{equation}\label{eq:assump1}
\tau_0\mathbf{I}_{3\times 3}\leq \epsilon_c(x)\leq \tau_0^{-1}\mathbf{I}_{3\times 3},\quad x\in\Sigma,
\end{equation}
where $0<\tau_0<1$ and $\mathbf{I}_{3\times 3}$ is the $3\times 3$ identity matrix. Fix $r_e<q<\sqrt{r_e^3}$ and let $\eta \rightarrow +0$. Let $k=k(\eta)$ be the smallest integer satisfying
\begin{equation}\label{eq:assump2}
r_e^{-k(\eta)}<\eta\leq r_e^{-k(\eta)+1},
\end{equation}
and let $\epsilon_s(x)$, $x\in B_{r_e}\backslash \overline{\Sigma}$ be given by
\begin{equation}\label{eq:assump3}
\epsilon_s=-1-k(\eta)^{-1},\quad k(\eta)\in\mathbb{N}.
\end{equation}
Consider the source $f$ given in \eqref{eq:source2}. We shall prove that

\begin{theorem}\label{thm:r2}
Consider the configuration $(\Omega; \epsilon_\eta)$ described in \eqref{eq:d2} with $(\Sigma; \epsilon_c)$ given in \eqref{eq:assump1} and $\epsilon_s$ in \eqref{eq:assump2}. Let the source $f$ be given by \eqref{eq:source2}. Let $q<r_e^{3/2}$ and assume that the source $f$ satisfies
\begin{equation}\label{eq:ca8}
k^{-1}|\alpha_k|^2\left(\frac{r_e^3}{q^2}\right)^{k}\rightarrow+\infty\quad\mbox{as}\ \ k\rightarrow+\infty,
\end{equation}
Then for any $M\in R_+$, there exists a sufficiently small $\eta=\eta(M)\in\mathbb{R}_+$ such that one has $E_\eta(u_\eta)\geq M$. That is, approximate ALR occurs.
\end{theorem}

\smallskip

\begin{remark}
Noting that $r_e^2<q^2<r_e^3$ and $r_e>1$, we see that the condition \eqref{eq:ca8} indicates that as long as the Fourier coefficient $\alpha_k$ of the source $f$ in \eqref{eq:source2} does not decay very quickly as $k\rightarrow +\infty$, then approximate ALR occurs.
\end{remark}

\smallskip

\begin{proof}[Proof of Theorem~\ref{thm:r2}]
We make use of the dual variational principle to construct a sequence $(v_\eta, \psi_\eta)$ satisfying
$$\nabla\cdot(\epsilon\nabla\psi_\eta)+\eta\Delta v_\eta=0 \quad \mbox{ with} \quad J_\eta(v_\eta, \psi_\eta)\rightarrow+\infty.$$
First, we set
\begin{equation}\label{eq:ca1}
\psi_\eta(x):=\lambda_\eta\Re \overline{\hat{\psi}_{k(\eta)}}(x),\quad x\in\mathbb{R}^3,
\end{equation}
where $\lambda_\eta$ is to be chosen below. Let $v_\eta\in H_{loc}^1(\mathbb{R}^3)$ be the solution to $\eta\Delta v_\eta=-\nabla\cdot(\epsilon\nabla\psi_\eta). $ By the standard elliptic estimate, one has
\begin{equation}\label{eq:ca2}
\eta\|\nabla v_\eta\|^2_{L^2(\mathbb{R}^3)}\leq C\eta^{-1}\|\nabla\cdot(\epsilon\nabla\psi_\eta)\|^2_{H^{-1}(\mathbb{R}^3)}\leq \widetilde{C}\eta^{-1}\lambda_\eta^2 k(\eta),
\end{equation}
where $C$ and $\widetilde{C}$ are two positive constants depending on $\Sigma$ and $\tau_0$ in \eqref{eq:assump1}. Next, by straightforward calculations, we have
\begin{equation}\label{eq:ca3}
\begin{split}
& E_\eta(u_\eta)\geq J_\eta(v_\eta, \psi_\eta)=\int_{\partial B_q} f\cdot \psi_\eta-\frac{\eta}{2}\int_{\mathbb{R}^3}|\nabla\psi_\eta|^2-\frac{\eta}{2}\int_{\mathbb{R}^3}|\nabla v_\eta|^2\\
\geq & \Re\lambda_\eta \int_{\partial B_q} f\cdot \overline{\hat{\psi}_{k(\eta)}}-\frac{\eta}{2}\int_{\mathbb{R}^3}|\nabla {\hat{\psi}_{k(\eta)}}|^2-\frac{\eta}{2}\int_{\mathbb{R}^3}|\nabla v_\eta|^2\\
\geq & \Re\int_{\partial B_q} \alpha_{k(\eta)} Y^{k(\eta)}\cdot \lambda_\eta q^{-k(\eta)} r_e^{2k(\eta)+1}\overline{Y^{k(\eta)}}\\
&\qquad -C\eta\lambda_\eta^2(2k(\eta)+1)r_e^{2k(\eta)+1}-C\eta^{-1}\lambda_\eta^2 k(\eta)\\
\geq & \widetilde{C} \Re \alpha_{k(\eta)}\lambda_\eta q^{-k(\eta)-1} r_e^{2k(\eta)+1}\\
&\qquad\quad-C\eta\lambda_\eta^2(2k(\eta)+1)r_e^{2k(\eta)+1}-C\eta^{-1}\lambda_\eta^2 k(\eta)\\
\geq & \lambda_\eta r_e^{k(\eta)}\left[\widetilde{C}\Re\alpha_{k(\eta)}\left(\frac{r_e}{q}\right)^{k(\eta)+1}-C\lambda_\eta(2k(\eta)+1)\eta r_{e}^{k(\eta)+1}-C\frac{\lambda_\eta}{\eta r_e^{k(\eta)}} k(\eta)   \right].
\end{split}
\end{equation}
By \eqref{eq:assump2}, we see that $1<\eta r_e^{k(\eta)}\leq r_e$ and $r_e<\eta r_e^{k(\eta)+1}\leq r_e^2$, and hence the last two terms in the last inequality in \eqref{eq:ca3} are of comparable order. Therefore, one further has from \eqref{eq:ca3} that
\begin{equation}\label{eq:ca4}
E_\eta(u_\eta)\geq \lambda_\eta r_e^{k_\eta}\left[\widetilde{C}\Re\alpha_{k(\eta)}\left(\frac{r_e}{q}\right)^{k(\eta)+1}-C\lambda_\eta k(\eta)\right].
\end{equation}
Using a completely similar argument by taking
\begin{equation}\label{eq:ca5}
\psi_\eta(x):=\lambda_\eta\Im \overline{\hat{\psi}_{k(\eta)}}(x),\quad x\in\mathbb{R}^3,
\end{equation}
one can show that
\begin{equation}\label{eq:ca6}
E_\eta(u_\eta)\geq \lambda_\eta r_e^{k_\eta}\left[\widetilde{C}\Im \alpha_{k(\eta)}\left(\frac{r_e}{q}\right)^{k(\eta)+1}-C\lambda_\eta k(\eta)\right].
\end{equation}
We choose $\lambda_\eta$ to be
\begin{equation}\label{eq:caaa7}
\lambda_\eta=\frac{\widetilde C}{2C k(\eta)}\Re\alpha_{k(\eta)}\left(\frac{r_e}{q}\right)^{k(\eta)+1},
\end{equation}
where $\Re\alpha_{k(\eta)}$ can be replaced by $\Im\alpha_{k(\eta)}$. Then one has
\begin{equation}\label{eq:ca7}
E_\eta(u_\eta)\geq \lambda_\eta r_e^{k(\eta)}\left[\frac 1 2 \widetilde{C}\Re\alpha_k \left(\frac{r_e}{q}\right)^{k(\eta)+1}\right]=\frac{1}{4 C k(\eta)}(\widetilde C\Re\alpha_{k(\eta)})^2\left(\frac{r_e}{q}\right)^2\left(\frac{r_e^3}{q^2}\right)^{k(\eta)}.
\end{equation}
Clearly, the estimate \eqref{eq:ca7} also holds with $\Re\alpha_{k(\eta)}$ replaced by $\Im \alpha_{k(\eta)}$. The proof can be immediately concluded by noting \eqref{eq:ca8} and \eqref{eq:ca7}.
\end{proof}

\subsection{Sensitivity and critical radius}

By Theorem~\ref{thm:r2}, we see that for any given $M\in\mathbb{R}_+$, one can determine a sufficiently small $\eta\in\mathbb{R}_+$ and a sufficiently large $k(\eta)\in\mathbb{N}$ according to \eqref{eq:assump2} and \eqref{eq:ca7}, such that the configuration with $\epsilon_s$ given in \eqref{eq:assump3} is ``almost" resonant in the sense that $E_\eta\geq M$. By \eqref{eq:ca8} and \eqref{eq:assump2}, one has that as $M\rightarrow+\infty$, $\eta\rightarrow +0$ and $k(\eta)\rightarrow+\infty$. Clearly, both $\eta$ and $k(\eta)$ depend on $M$, and hence $\epsilon_s$ of the plasmonic configuration depends on $M$ as well. It is natural to ask what would happen if one fixes the integer $k(\eta)$ in \eqref{eq:assump3}. That is, $k(\eta)$ in \eqref{eq:assump3} is replaced by an integer $k_0$, which can be as large as possible, but fixed. Next we show that resonance does not occur in such a case, and this indicates that the resonance is very sensitive to the plasmonic parameter.

\smallskip

\begin{theorem}\label{thm:sensitivity}
Let $k_0\in\mathbb{N}$ be any fixed positive integer and let $\Sigma=B_1$. Let $(B_{r_e};\epsilon_\eta)$ be given in \eqref{eq:d2} with $\epsilon_c=1$ and $\epsilon_s=-1-k_0^{-1}$. Suppose that the source $f$ is given in \eqref{eq:source2}. Then ALR does not occur.
\end{theorem}

\smallskip

Before giving the proof of Theorem~\ref{thm:sensitivity}, we present another theorem whose proof would be more general than the one needed for Theorem~\ref{thm:sensitivity}.

\smallskip

\begin{theorem}\label{thm:crc}
Let $(B_{r_e};\epsilon_\eta)$ be given in \eqref{eq:d2} with $\Sigma=B_1$ and $\epsilon_c=1$. Suppose that $f$ is given in \eqref{eq:source2} and $\eta\rightarrow+0$. Let $k(\eta)$ be chosen according to \eqref{eq:assump2}. Then if $q>r_*$ with $r_*$ given in \eqref{eq:critical}, then ALR does not occur.
\end{theorem}

\smallskip

By Theorem~\ref{thm:r2}, we see that ALR occurs if $\epsilon_s$ is chosen according to \eqref{eq:assump3}, namely, it is variable depending on the asymptotic parameter $\eta$, and the source is located within the critical radius $r_*$ and satisfies the generic condition \eqref{eq:ca8}. However, by Theorem~\ref{thm:sensitivity}, it is pointed out that the resonance phenomenon is very sensitive with respect to the plasmonic parameter $\epsilon_s$, and if it is independent of the asymptotic parameter $\eta$, then resonance does not occur. Theorem~\ref{thm:crc} further shows that for the case with the variable plasmon parameter in Theorem~\ref{thm:r2}, the resonance phenomenon is localized.

\medskip

\begin{proof}[Proof of Theorem~\ref{thm:crc}]
We make use of the primal variational principle to prove the theorem. To that end, we first construct test functions $v_\eta$ and $w_\eta$ that satisfy the PDE constraint in \eqref{eq:pv}.

Let $k(\eta)\in \mathbb{N}$ be such that $\eta\approx r_e^{-k(\eta)}$. Let $v_\eta\in H_{loc}^1(\mathbb{R}^3)$ be of the following form
\begin{equation}\label{eq:crc1}
v_\eta=\sum_{k\neq k(\eta)} v_{\eta, k}+v_{\eta, k(\eta)},
\end{equation}
where $v_{\eta, k}$, $k\neq k(\eta)$, satisfies
\begin{equation}\label{eq:crc2}
\nabla\cdot(\epsilon\nabla v_{\eta, k})=\alpha_k f_k^q,
\end{equation}
and $v_{\eta, k(\eta)}$ satisfies
\begin{equation}\label{eq:crc3}
\nabla\cdot(\epsilon\nabla v_{\eta, k(\eta)})=\alpha_{k(\eta)} f_{k(\eta)}^q\quad\mbox{on\ \ $\partial B_q$}.
\end{equation}
Define $\hat{v}_k$ to be
\begin{equation}\label{eq:crc4}
\hat{v}_k(x):=\begin{cases}
\ \ r^k Y^k(\hat x),\qquad &\  |x|\leq 1,\\
\bigg(\frac{1+ k + k(\eta)}{(1+2k)(1+k(\eta))}r^k + \frac{k+2 k k(\eta)}{(1+2k)(1+k(\eta))}r^{-k-1} \bigg)Y^k(\hat x), & \ 1< \abs{x} \leq r_e, \\
\bigg(\frac{r_e^{-2 k-1} \left(\left(-k (k+1)+k(\eta)^2+k(\eta)\right) r_e^{2 k+1}+k (k+1) (2 k(\eta)+1)^2\right)}{(2 k+1)^2 k(\eta)(k(\eta)+1)}r^k\\
\qquad\quad +\frac{k (2 k(\eta)+1) (k+k(\eta)+1) (r_e^{2 k+1}-1)}{(2 k+1)^2 k(\eta) (k(\eta)+1)}r^{-k-1} \bigg)Y^k(\hat x), &\ r_e< \abs{x} \leq q,\\
\bigg( \frac{r_e^{-2 k-1} \big[-(k+k(\eta)+1) r_e^{2 k+1} \left((k-k(\eta)) q^{2 k+1}+2 k k(\eta)+k\right)
         +k (k+1) (2 k(\eta)+1)^2 q^{2 k+1} }{(2 k+1)^2 k(\eta) (k(\eta)+1)}    \\
 \qquad\quad+\frac{k (2 k(\eta)+1) (k+k(\eta)+1) r_e^{4 k+2}\big]}{(2 k+1)^2 k(\eta) (k(\eta)+1)} \bigg) r^{-k-1} Y^k(\hat x), & \ |x|>q.
\end{cases}
\end{equation}
Using $\hat v_k$ in \eqref{eq:crc4}, we then set
\begin{equation}\label{eq:crc5}
v_{\eta, k}=\lambda_k\hat{v}_k,\qquad k\neq k(\eta),
\end{equation}
where
\begin{equation}\label{eq:crc6}
\lambda _{k}= \alpha_k
 \frac{(2 k+1) k(\eta) (k(\eta)+1)} {q^{k-1} r_e^{-2 k-1} \big((k-k(\eta)) (k+k(\eta)+1) r_e^{2 k+1}-k (k+1) (2 k(\eta)+1)^2\big)}.
\end{equation}
By straightforward calculations, though a bit tedious, one can verify that $v_{\eta, k}$ defined in \eqref{eq:crc6} satisfies \eqref{eq:crc2}. Next, we define
\[
\widetilde{V}_{k(\eta)}(x)=\begin{cases}
r^{k(\eta)} Y^{k(\eta)}(\hat x),\qquad &\quad |x|\leq q,\\
r^{-k(\eta)-1} q^{2k(\eta)+1} Y^{k(\eta)}(\hat x),\qquad &\quad |x|>q,
\end{cases}
\]
and set
\begin{equation}\label{eq:crc7}
v_{\eta,k(\eta)}=\lambda_{k(\eta)}\widetilde{V}_{k(\eta)},\qquad \lambda_{k(\eta)}=-\alpha_{k(\eta)}\frac{q^{1-k(\eta)}}{2k(\eta)+1}.
\end{equation}
It is directly verified that $v_{\eta, k(\eta)}$ satisfies \eqref{eq:crc3}. Finally, we set
\begin{equation} \label{eq:crc8}
\begin{split}
 -\Delta w_{\eta}&= -\nabla\cdot(\epsilon\nabla v_\eta)+f\\
  & = -\nabla \cdot (\epsilon \nabla v_{\eta, k(\eta)}) +\alpha_{k(\eta)} f_{k(\eta)}^q \\
 & = -\lambda_{k(\eta)}[\nu\cdot\epsilon\nabla\widetilde{V}_{k(\eta)}]\big|_{\partial B_1}\mathcal{H}^2\lfloor\partial B_1-\lambda_{k(\eta)}[\nu\cdot\epsilon\nabla \widetilde{V}_{k(\eta)}]\big|_{\partial B_{r_e}}\mathcal{H}^2\lfloor\partial B_{r_e}
\end{split}
\end{equation}
Clearly, $\Re v_\eta$ and $\Re w_\eta$ satisfy the PDE constraint in \eqref{eq:pv}, and hence by the primal variational principle,
\begin{equation}\label{eq:crc9}
E_\eta(u_\eta)\leq I_\eta(\Re v_\eta, \Re w_\eta)\leq I_\eta(v_\eta, w_\eta),
\end{equation}
where $I_\eta$ is defined in \eqref{eq:v1}.

We proceed to calculate the energy $I_\eta(v_\eta, w_\eta)$ in \eqref{eq:crc9} and show that it is bounded as $\eta\rightarrow +0$, which readily implies that ALR does not occur. First, by \eqref{eq:crc6}, one can verify that
\begin{equation}
\abs{\lambda_k} \leq C \abs{\alpha_k} \frac{\left(k(\eta)\right)^6}{q^k k}, \quad k\neq k(\eta).
\end{equation}
Hence we have the following estimate
\begin{equation} \label{eq:crc10}
\begin{split}
\eta \int_{\mathbb{R}^3}{ \abs{\nabla v_{\eta,k}}^2 } =& \eta  \abs{\lambda_k}^2 \int_{\mathbb{R}^3}{\abs{ \nabla \hat{v}_k}^2  }\\
\leq&  C \eta  \abs{\alpha_k}^2 \left(k(\eta)\right)^{12} \frac{1}{q^{2k} k^2} k q^{2k}
 = C \eta \abs{\alpha_k}^2 \frac{\left(k(\eta)\right)^{12} }{k}.
\end{split}
\end{equation}
Since $r_e^{-k(\eta)} \approx \eta$, we have $\eta \left(k(\eta)\right)^{12} \leq C$, which together with \eqref{eq:crc10} implies that
\begin{equation}\label{eq:crc11}
\eta  \sum_{k \neq k(\eta)}  \int_{\mathbb{R}^3}{  \abs{\nabla v_{\eta,k}}^2 } \leq  C \sum_{k \neq k(\eta)} \abs{\alpha_k}^2 \leq C\|F\|_{L^2(\partial B_q)}^2.
\end{equation}
By \eqref{eq:crc7}, one can also calculate that
\begin{equation}\label{eq:crc12}
 \eta \int_{\mathbb{R}^3}{ \abs{\nabla v_{\eta,k(\eta)}}^2 } \leq  C \eta \abs{\lambda_{k(\eta)}}^2 \int_{\mathbb{R}^3}{\abs{\nabla  \widetilde{V}_{k(\eta)} }^2  }
  \leq  C \eta \abs{\alpha_{k(\eta)}}^2
\end{equation}
Next we estimate the energy due to $w_{\eta}$, and by \eqref{eq:crc8} one has
\begin{equation}\label{eq:crc13}
 \frac 1 \eta\int_{\mathbb{R}^3}|\nabla w_\eta|^2\leq C\frac 1 \eta\|\nabla\cdot(\epsilon \nabla v_{\eta, k(\eta)})-\alpha_{k(\eta)} f_{k(\eta)}^q\|_{H^{-1}}^2\leq C\frac 1 \eta|\lambda_{k(\eta)}|^2 r_e^{2k(\eta)}k(\eta),
\end{equation}
where we have made use of the fact in \eqref{eq:crc8} that $ \nabla \cdot (\epsilon \nabla v_{\eta, k(\eta)}) -\alpha_{k(\eta)} f_{k(\eta)}^q$ is supported on $\abs{x}=1$
and $\abs{x}=r_e$.
Now by the choice of $\lambda_{k(\eta)}$ from (\ref{eq:crc7}), we have $\abs{\lambda_{k(\eta)}} \leq C \abs{\alpha_{k(\eta)}}  \left(k(\eta)\right)^{-1}q^{-k(\eta)}$, and hence
\begin{equation}\label{e10}
 \frac{1}{\eta} \int_{\mathbb{R}^3}{ \abs{\nabla w_{\eta}}^2 } \leq   \frac{C}{\eta} \abs{\alpha_{k(\eta)}}^2   (\frac{r_e}{q})^{2k(\eta)}
  \leq   C \abs{\alpha_{k(\eta)}}^2 (\frac{r_e^{ 3/2  }}{q})^{2k(\eta)},
\end{equation}
where the last inequality follows from the assumption that $ r_e^{-k(\eta)} \approx \eta$.

By combining \eqref{eq:crc11}, \eqref{eq:crc12} and \eqref{e10}, one readily sees that $I_\eta(v_\eta, w_\eta)$ remains bounded as $\eta\rightarrow +0$, and thus completes the proof.
\end{proof}

\medskip

\begin{proof}[Proof of Theorem~\ref{thm:sensitivity}]
The proof follows from a completely similar argument to that in the proof of Theorem~\ref{thm:crc}. Indeed, by the primal variational principle, one can construct test functions $w_\eta\equiv 0$, and $v_\eta$ satisfying $\nabla\cdot(\epsilon\nabla v_\eta)=f$ of the form \eqref{eq:crc1} with $v_{\eta,k}$ given in \eqref{eq:crc5} for all $k\in\mathbb{N}$. Then by similar estimates in deriving \eqref{eq:crc11}, one can show that $I_\eta(v_\eta, 0)$ remains bounded as $\eta\rightarrow +0$, and thus completes the proof.

\end{proof}

Throughout this section, we have mainly considered the resonance aspect, namely \eqref{eq:m3}, for the introduced plasmonic configurations. The nonoccurrence of ALR obviously implies the nonoccurrence of CALR; see Theorems~\ref{thm:nr1}, \ref{thm:sensitivity} and \ref{thm:crc}. However, for the resonance results derived in Theorems~\ref{thm:r1} and \ref{thm:r2}, in order to show the occurrence of CALR or not, one needs further derive the rate at which the energy blows up to infinity in terms of the asymptotic parameter $\eta$. It is also remarked that for our study in the present section, we have assumed that the loss parameter $\eta$ is given over the whole space $\mathbb{R}^3$; see \eqref{eq:d2}. It would be more practical to consider the case that the loss parameters presented inside and outside the cloaking device are different. We shall investigate those interesting issues in our future study. On the other hand, it is emphasized that in our study in the next section, i.e. Section~\ref{sect:3}, we shall actually consider CALR and the case that the loss parameter is only presented within the cloaking device.

\section{CALR in three dimensions}\label{sect:3}

In this section, we consider a plasmonic construction as follows
\begin{equation}\label{eq:np1}
\epsilon_\eta(x)=\begin{cases}
\epsilon_c,\qquad & |x|\leq r_i,\medskip\\
\epsilon_s+i\eta,\qquad & r_i<|x|\leq r_e,\medskip\\
\epsilon_m,\qquad & r_e<|x|,
\end{cases}
\end{equation}
with
\begin{equation}\label{eq:np2}
\epsilon_c=(1+\frac{1}{k_0})^2,\quad \epsilon_s=-1-\frac{1}{k_0},\quad\epsilon_m=1,
\end{equation}
where $k_0\in\mathbb{N}$ will be properly chosen in what follows. It is remarked that for the occurrence of CALR, $k_0$ will eventually depend on $\eta$ (cf. \eqref{eq:k0}). Our argument shall follow a similar spirit to that in \cite{Ack13}.

Throughout the present section, we let $R\in\mathbb{R}_+$ and $R>r_e$, and $f\in L^2(\mathbb{R}^3)$ be compactly supported in $\mathbb{R}^3\backslash\overline{B_R}$, satisfying \eqref{eq:m2}. It is remarked that compared to our earlier study in Section~\ref{sect:2}, we only assume loss in the plasmon layer, and the source considered would be more general. In the sequel, we let $F$ be defined by
\begin{equation}\label{eq:np3}
F(x)=\int_{\mathbb{R}^3} G(|x-y|) f(y)\ dy,\quad G(t)=-\frac{1}{4\pi t},\quad x\in\mathbb{R}^3,
\end{equation}
which is the Newtonian potential of $f$. Note that $\Delta F(x)=f(x)$, and hence $F(x)$ is harmonic in $r_e<|x|<R$, and it can be expressed as
\begin{equation}\label{eq:newtonpotential}
F(x)=\sum_{k=0}^{\infty}\sum_{l=-k}^{k} \beta_k^l \abs{x}^k Y_k^l(\hat{x}),\quad r_e<|x|<R.
\end{equation}
Then the solution $u_\eta$ to (\ref{eq:m1}) with $\epsilon_\eta$ in \eqref{eq:np1} can be expressed in $\abs{x} < R$ as follows,
\begin{equation}\label{eq:np4}
  u_\eta(x)=  \left \{
       \begin{array}{c c }
         \displaystyle{u_c(x)=\sum_{k=0}^{\infty}\sum_{l=-k}^{k} a_{k}^{l}\abs{x}^kY_k^l(\hat{x}) , }     & \abs{x} \leq r_i,\\
       \displaystyle{  u_s(x)=\sum_{k=0}^{\infty}\sum_{l=-k}^{k} (b_{k}^{l}\abs{x}^k + c_{k}^{l}\abs{x}^{-k-1} )Y_k^l(\hat{x}) , }               & r_i < \abs{x} \leq r_e, \\
       \displaystyle{  u_m(x)=\sum_{k=0}^{\infty}\sum_{l=-k}^{k} (e_{k}^{l}\abs{x}^k + d_{k}^{l}\abs{x}^{-k-1} )Y_k^l(\hat{x})  ,   }           & r_e < \abs{x} \leq R .\\
       \end{array}
       \right.
\end{equation}
Using the transmission conditions across the interfaces $|x|=r_i, r_e$ and $R$, respectively, one has by straightforward calculations (though a bit lengthy and tedious) that
\begin{equation}\label{eq:np5}
 a_k^l=a_k e_k^l, \quad b_k^l=b_k e_k^l, \quad c_k^l=c_k e_k^l, \quad d_k^l=d_k e_k^l
\end{equation}
with
\begin{align}
a_k=& \left[-(2k+1)^2 \epsilon_m( \epsilon_s+i \eta)\right ]\times \bigg[ \rho^{2k+1}(k^2+k)(( \epsilon_s+i \eta)-\epsilon_c)\times \nonumber\\
& (( \epsilon_s+i \eta)-\epsilon_m)-( (k+1)( \epsilon_s+i \eta) + k\epsilon_c ) ( (k+1)\epsilon_m + k( \epsilon_s+i \eta) )\bigg ]^{-1},\label{eq:coeffa}\\
b_k=& \left[ -\epsilon_m (2k+1)( (k+1)( \epsilon_s+i \eta) + k\epsilon_c) \right]\times \bigg[ \rho^{2k+1}(k^2+k)(( \epsilon_s+i \eta)-\epsilon_c)\times \nonumber \\
&(( \epsilon_s+i \eta)-\epsilon_m) - ( (k+1)( \epsilon_s+i \eta) + k\epsilon_c ) ( (k+1)\epsilon_m + k( \epsilon_s+i \eta) ) \bigg]^{-1},\label{eq:coeffb}\\
c_k=& \left [-r_i^{2k+1} \epsilon_m k(2k+1)(( \epsilon_s+i \eta)-\epsilon_c)\right ]\times \bigg[ \rho^{2k+1}(k^2+k)(( \epsilon_s+i \eta)-\epsilon_c)\times\nonumber \\
& (( \epsilon_s+i \eta)-\epsilon_m) - ( (k+1)( \epsilon_s+i \eta) + k\epsilon_c ) ( (k+1)\epsilon_m + k( \epsilon_s+i \eta) )\bigg]^{-1},\label{eq:coeffc}\\
d_k=& \bigg[kr_e^{2k+1}[ (\epsilon_m-( \epsilon_s+i \eta))( (k+1)( \epsilon_s+i \eta) +k\epsilon_c )+\nonumber\\
& \qquad\qquad \rho^{2k+1} (( \epsilon_s+i \eta) -\epsilon_c)(k\epsilon_m+(k+1)( \epsilon_s+i \eta)) ]\bigg]\times\nonumber\\
& \bigg[ \rho^{2k+1}(k^2+k)(( \epsilon_s+i \eta)-\epsilon_c) (( \epsilon_s+i \eta)-\epsilon_m)\nonumber \\
&- ( (k+1)( \epsilon_s+i \eta) + k\epsilon_c ) ( (k+1)\epsilon_m + k( \epsilon_s+i \eta) )\bigg]^{-1},\label{eq:coeffd}
\end{align}
where and also in what follows, $\rho=r_i/r_e$.
Since $u_\eta-F$ is harmonic in $\abs{x}> r_e$ and tends to $0$ as $\abs{x} \rightarrow +\infty$, one must have
\begin{equation}\label{eq:np6}
e_k^l=\beta_k^l.
\end{equation}
Hence the solution in the shell, namely $u_s$, is given by
\begin{equation}\label{eq:np7}
 u_s(x)=\sum_{k=0}^{\infty}\sum_{l=-k}^{k} \beta_k^l (b_{k}\abs{x}^k + c_{k}\abs{x}^{-k-1} )Y_k^l(\hat{x})
\end{equation}
Using Green's identity and the orthogonality of $Y_k^l$, we further obtain from \eqref{eq:np7} that
\begin{equation}\label{eq:np8}
\int_{r_i < \abs{x} < r_e} \abs{\nabla u_s(x)}^2 \approx \ \sum_{k=0}^{\infty}\sum_{l=-k}^{k} k \abs{\beta_k^l}^2 (\abs{b_k}^2 r_e^{2k} + \abs{c_k}^2 r_i^{-2k-1})
\end{equation}
Therefore,
\begin{equation}\label{eq:np9}
  E_{\eta} \approx \eta \sum_{k=0}^{\infty}\sum_{l=-k}^{k} k \abs{\beta_k^l}^2 (\abs{b_k}^2 r_e^{2k} + \abs{c_k}^2 r_i^{-2k-1}).
\end{equation}

We are ready to present the results on the occurrence and nonoccurrence of CALR. In what follows, similar to \eqref{eq:critical}, we set
\begin{equation}\label{eq:crs}
r_*=\sqrt{r_e^3/r_i}.
\end{equation}

\begin{theorem}\label{thm:s1}
Let $\epsilon_\eta$ be given in \eqref{eq:np1}-\eqref{eq:np2}, and let $f$ be as described above, supported in $R<|x|<r_*$ for some $R\in (r_e, r_*)$. Then there exists an infinite subsequence $\{n_j\}\subset\mathbb{N}$ which satisfies $n_j\rightarrow+\infty$ as $\eta\rightarrow+0$ such that as $k_0=n_j$, then
\begin{equation}\label{eq:np10}
E_\eta(u_\eta^j)\rightarrow +\infty\quad\mbox{as}\quad \eta\rightarrow+0,
\end{equation}
where $u_\eta^j$ denotes the solution corresponding to $n_j$. Moreover, $u_\eta^j(x)$ remains bounded for $|x|>r_*$, and hence CALR occurs.
\end{theorem}
\begin{proof}
Using \eqref{eq:np2}, we first have by straightforward calculations that
\begin{equation}\label{eq:np11}
\begin{split}
  &\bigg|\rho^{2k_0+1}(k_0^2+k_0)(( \epsilon_s+i \eta)-\epsilon_c) (( \epsilon_s+i \eta)-\epsilon_m)\\
  &\qquad\qquad - ( (k_0+1)( \epsilon_s+i \eta) + k\epsilon_c ) ( (k_0+1)\epsilon_m + k_0( \epsilon_s+i \eta) )\bigg| \\
  &  \approx k_0^2(\eta^2 + \rho^{2k_0}),
\end{split}
\end{equation}
which together with \eqref{eq:coeffa} and \eqref{eq:coeffb} yields that
\begin{equation}\label{eq:np12}
\abs{b_{k_0}} \approx  \frac{\eta} {\eta^2 + \rho^{2k_0}}, \quad   \abs{c_{k_0}} \approx  \frac{r_i^{2k_0}} {\eta^2 + \rho^{2k_0}}.
\end{equation}
By applying \eqref{eq:np12} to \eqref{eq:np9}, we then obtain
\begin{equation}\label{eq:np13}
 E_{\eta} \geq \sum_{m=-k_0}^{k_0} \frac{\eta k_0 r_e^{2k_0} \abs{\beta_{k_0}^m}^2 }{\eta^2 + \rho^{2k_0}}.
\end{equation}
Let $k_0$ be chosen such that
\begin{equation}\label{eq:k0}
 \rho^{k_0} <\eta \leq \rho^{k_0-1},
\end{equation}
 and hence
\begin{equation}\label{eq:np14}
 E_{\eta} \geq  Ck_0 \frac{r_e ^{3k_0}}{r_i^{k_0}} \sum_{l=-k_0}^{k_0} \abs{\beta_{k_0}^l}^2 \geq C\frac{r_e ^{3k_0}}{r_i^{k_0}} \frac{k_0}{2k_0+1}    (\sum_{l=-k_0}^{k_0} \abs{\beta_{k_0}^l} )^2   \geq C \frac{r_e ^{3k_0}}{r_i^{k_0}}  (\sum_{l=-k_0}^{k_0} \abs{\beta_{k_0}^l} )^2.
\end{equation}

Since the source $f$ is supported inside the critical radius $r_* = \sqrt{\frac{r_e ^{3}}{r_i}}$ and its
Newtonian potential $F$ cannot be harmonically extended  into $\abs{x} < r_*$, one can see that there holds
\begin{equation}\label{eq:np15}
 \limsup_{k \rightarrow \infty} (\sum_{l=-k}^{k} \abs{\beta_k^l} )^{1/k} > 1/\sqrt{\frac{r_e ^{3}}{r_i}},
\end{equation}
which together with \eqref{eq:np14} readily implies the existence of a subsequence $\{n_j\}\subset\mathbb{N}$ that fulfills the properties stated in the theorem.

Finally, we prove that $u_\eta$ is bounded in the region with $|x|>r_*$.
For $k_0$ in \eqref{eq:k0}, we have from \eqref{eq:coeffd} that
\begin{equation}\label{eq:coffd2}
 \abs{d_{k_0}} \approx \frac{r_e^{2k_0} \abs{i k_0 \eta +\rho^{2k_0}}}{k_0(\rho^{2k_0} + \eta^2)}  \leq C\frac{r_e^{3k_0}}{r_i^{k_0}},
\end{equation}
and for $k \neq k_0$,
\begin{equation}\label{eq:coffd2n}
  \abs{d_k}  \leq C k r_e^{2k} \frac{\abs{1-k/k_0} + \rho^{2k} }{k^2 \rho^{2k}  + \abs{1-k/k_0}^2 } \leq
  C k r_e^{2k} (\frac{1}{k \rho^k} + \frac{1}{k^2} ) \leq C\frac{r_e^{3k}}{r_i^{k}}.
\end{equation}
By using \eqref{eq:np4}, \eqref{eq:coffd2} and \eqref{eq:coffd2n}, one then calculates for $\abs{x} = r >r_e^2 r_i^{-1}$ that
\begin{equation}
 \abs{u_m(x) - F(x)} \leq C \sum_{k=0}^{\infty}\sum_{l=-k}^{k}  \abs{\beta_k^l} \frac{r_e^{3k}}{r_i^k} \abs{x}^{-k-1} < +\infty.
\end{equation}

The proof is complete.
\end{proof}

\medskip

Similar to Theorem~\ref{thm:r2}, using Theorem~\ref{thm:s1}, one can show that by using variable material parameters $\epsilon_c$ and $\epsilon_s$ in \eqref{eq:np2}, depending only on the asymptotic parameter $\eta$, CALR will occur. Next, we shall show the sensitivity and localization feature of the resonance, similar to Theorems~\ref{thm:sensitivity} and \ref{thm:crc}.

\medskip

\begin{theorem}\label{thm:s2}
Let $\epsilon_\eta$ be given in \eqref{eq:np1}-\eqref{eq:np2} with $k_0\in\mathbb{N}$ be any fixed integer, and let $f\in L^2(\mathbb{R}^3)$ be compactly supported in $|x|>R>r_e$ and satisfy \eqref{eq:m2}. Then ALR does not occur.
\end{theorem}

\begin{proof}
By using \eqref{eq:np12}, one has
\begin{equation}\label{eq:eee1}
\begin{split}
& \sum_{l=-k_0}^{k_0} k_0 \abs{\beta_{k_0}^l}^2 (\abs{b_{k_0}}^2 r_e^{2k_0} + \abs{c_{k_0}}^2 r_i^{-2k_0-1})\\
\leq & C \sum_{l=-k_0}^{k_0} \frac{\eta k_0 r_e^{2k_0} \abs{\beta_{k_0}^l}^2 }{\eta^2 + \rho^{2k_0}} \leq
C \sum_{l=-k_0}^{k_0} \frac{ \eta k_0 r_e^{4k_0} \abs{\beta_{k_0}^l}^2 }{ r_i^{2 k_0}} \leq C\eta\|f\|_{L^2(\mathbb{R}^3)}^2.
\end{split}
\end{equation}
For $k\neq k_0$, we have by \eqref{eq:coeffb} and \eqref{eq:coeffc} that
\begin{equation}\label{eq:eee2}
|b_k|\leq C k^2\quad\mbox{and}\quad |c_k|\leq C k^2 r_i^{2k},
\end{equation}
and hence their contribution to $E_\eta$ is
\begin{equation}\label{eq:eee3}
\begin{split}
&\eta \sum_{k \neq k_0} \sum_{l=-k}^{k} k \abs{\beta_k^l}^2 (\abs{b_k}^2 r_e^{2k} + \abs{c_k}^2 r_i^{-2k-1})\\
 \leq &  C \eta \sum_{k \neq k_0} \sum_{l=-k}^{k}  k^3 \abs{\beta_k^l}^2  r_e^{2k} \leq C\eta\|f\|_{L^2(\mathbb{R}^3)}.
\end{split}
\end{equation}
The proof can be completed by combining \eqref{eq:np9}, \eqref{eq:eee2} and \eqref{eq:eee3}.
\end{proof}

\begin{theorem}\label{thm:s3}
Let $\epsilon_\eta$ be given in \eqref{eq:np1}-\eqref{eq:np2} with $k_0=k(\eta)\in\mathbb{N}$ chosen according to \eqref{eq:k0}, and let $f\in L^2(\mathbb{R}^3)$ be compactly supported in $|x|>r_*$ and satisfy \eqref{eq:m2}. Then ALR does not occur.
\end{theorem}

\begin{proof}
We calculate the energy $E_\eta$ given in \eqref{eq:np9}. For $k=k(\eta)$, similar to \eqref{eq:eee1}, along with the use of the fact that $\rho^{k(\eta)}\approx \eta$, one has
\begin{equation}\label{eq:eee4}
\begin{split}
& \sum_{l=-k(\eta)}^{k(\eta)} k(\eta) \abs{\beta_{k(\eta)}^l}^2 \big(\abs{b_{k(\eta)}}^2 r_e^{2k(\eta)} + \abs{c_{k(\eta)}}^2 r_i^{-2k(\eta)-1}\big)\\
\leq & C \sum_{l=-k(\eta)}^{k(\eta)} \frac{\eta k(\eta) r_e^{2k(\eta)} \abs{\beta_{k(\eta)}^l}^2 }{\eta^2 + \rho^{2k(\eta)}} \leq
C \sum_{l=-k(\eta)}^{k(\eta)} \frac{ k(\eta) r_e^{3k(\eta)} \abs{\beta_{k(\eta)}^l}^2 }{ r_i^{k(\eta)}}.
\end{split}
\end{equation}
For $k\neq k(\eta)$, by \eqref{eq:coeffb} and \eqref{eq:coeffc}, one has
\[
|b_k| \leq C k^2 k(\eta)^2\quad \mbox{and}  \quad |c_k| \leq C k^2 k(\eta)^2 r_i^{2k},
\]
and hence their contribution to $E_{\eta}$ is
\begin{equation}\label{eq:eee5}
\begin{split}
& \eta \sum_{k \neq k(\eta)} \sum_{l=-k}^{k} k \abs{\beta_k^l}^2 (\abs{b_k}^2 r_e^{2k} + \abs{c_k}^2 r_i^{-2k-1})\\
  \leq &  C \eta \sum_{k \neq k(\eta)} \sum_{l=-k}^{k}  k^3 k(\eta)^2 \abs{\beta_k^l}^2  r_e^{2k} \leq C \sum_{k \neq k(\eta)} \sum_{l=-k}^{k}  \frac{ k r_e^{3k} \abs{\beta_k^l}^2 }{ r_i^{k}},
\end{split}
\end{equation}
where in the last inequality we have made use of the fact that
\begin{equation}
k^2 < C \frac{r_e^k}{r_i^k}\quad\mbox{and}  \quad  k(\eta)^2  \eta <C.
\end{equation}
If the source function $f$ is supported outside the sphere of critical radius $r_* = \sqrt{\frac{r_e ^{3}}{r_i}}$, then
its Newtonian potential $F$ in \eqref{eq:newtonpotential} can be harmonically extended into $\abs{x}< r_* + 2\tau$ for a sufficiently small $\tau\in\mathbb{R}_+$. Hence, by combining \eqref{eq:eee4} and \eqref{eq:eee5}, we have
\begin{equation}
E_{\eta} \leq C  \sum_{k=0}^{\infty}\sum_{l=-k}^{k}  k \abs{\beta_k^l}^2  \frac{r_e^{3k}}{r_i^{k}} <
 C \|{f}\|_{L^2(\mathbb{R}^3)}^2 <+\infty
\end{equation}
which readily completes the proof.
\end{proof}

\section*{Acknowledgement}

The authors would like to express their gratitudes to the two anonymous referees for their insightful and constructive comments, which have led to significant improvement on the results and presentation of this work.
The work of H. Li was supported by the NSF of China, No.\,11101033. The work of J. Li was partially supported by the NSF of China, No.\,11201453 and the Shenzhen Research Program, No.\,JCYJ20140509143748226.
The work of H. Liu was supported by the FRG and start-up grants of Hong Kong Baptist University, and the NSF grant of China, No.\,11371115.

\end{document}